\documentclass[12pt,twoside]{amsart}
\usepackage{amsmath, amsthm, amscd, amsfonts, amssymb, graphicx}
\usepackage[bookmarksnumbered, plainpages]{hyperref}

\newtheorem{thm}{Theorem}[section]

\newtheorem{lem}[thm]{Lemma}
\newtheorem{prop}[thm]{Proposition}
\newtheorem{defn}[thm]{Definition}

\numberwithin{equation}{section}

\begin{document}

\title{\bf Gauss-Bonnet theorems in the affine group and the group of rigid motions of the Minkowski plane}

\author{Yong Wang, Sining Wei}

\thanks{{\scriptsize
\hskip -0.4 true cm \textit{2010 Mathematics Subject Classification:}
53C40; 53C42.
\newline \textit{Corresponding author:} Yong Wang}}

\maketitle

\begin{abstract}
 In this paper, we compute sub-Riemannian limits of Gaussian curvature for a Euclidean $C^2$-smooth surface in the affine group and the group of rigid motions of the Minkowski plane away from
 characteristic points and signed geodesic curvature
 for Euclidean $C^2$-smooth curves on surfaces. We get Gauss-Bonnet theorems in the affine group and the group of rigid motions of the Minkowski plane.

\end{abstract}

\vskip 0.2 true cm


\pagestyle{myheadings}
\markboth{\rightline {\scriptsize Wang and Wei}}
         {\leftline{\scriptsize Gauss Bonnet theorems}}

\bigskip
\bigskip


\section{ Introduction}
 \indent In \cite{DV}, Gaussian curvature for non-horizontal surfaces in sub-Riemannian Heisenberg space $\mathbb{H}^1$ was
defined and a Gauss-Bonnet theorem was proved. The definition was analogous to Gauss curvature
of surfaces in $\mathbb{R}^3$ with particular normal to surface and Hausdorff measure of area. The image of
Gauss map was in the cylinder of radius one. In \cite{BTV}, Balogh-Tyson-Vecchi used a Riemannnian approximation scheme to define a notion of intrinsic
Gaussian curvature for a Euclidean $C^2$-smooth surface in the Heisenberg group $\mathbb{H}^1$ away
from characteristic points, and a notion of intrinsic signed geodesic curvature for Euclidean
$C^2$-smooth curves on surfaces. These results were then used to prove a Heisenberg version of
the Gauss-Bonnet theorem. In \cite{Ve}, Veloso verified that Gausssian curvature of surfaces and normal curvature of curves in surfaces introduced by \cite{DV} and by \cite{BTV} to prove Gauss-Bonnet theorems
in Heisenberg space $\mathbb{H}^1$ were unequal and he applied the same formalism of \cite{DV} to
get the curvatures of \cite{BTV}. With the obtained formulas, it is possible to prove the Gauss-Bonnet
theorem in \cite{BTV} as a straightforward application of Stokes theorem. \\
\indent The Riemannian approximation scheme used in \cite{BTV} ,
can in general depend upon the choice of the complement to the horizontal distribution.
In the context of $\mathbb{H}^1$ the choice which they have adopted is rather natural. The existence of the limit defining the intrinsic curvature of a surface
depends crucially on the cancellation of certain divergent quantities in the limit.
Such cancellation stems from the specific choice of the adapted frame
bundle on the surface, and on symmetries of the underlying left-invariant group structure
on the Heisenberg group. In \cite{BTV}, they proposed an interesting question to understand to what extent similar
phenomena hold in other sub-Riemannian geometric structures. In this paper, we solve this problem for the affine group and the group of rigid motions of the Minkowski plane.
In the case of affine group, the cancellation of certain divergent quantities in the limit does not happen and the limit of the Riemannian Gaussian curvature is divergent.
In the case of group of rigid motions of the Minkowski plane, similarly to the Heisenberg group, the cancellation of certain divergent quantities in the limit happens and
the limit of the Riemannian Gaussian curvature exists. We also get Gauss-Bonnet theorems in the affine group and the group of rigid motions of the Minkowski plane.

\indent In Section 2, we compute the sub-Riemannian limit of curvature of curves in the affine group. In Section 3, we compute sub-Riemannian limits of geodesic curvature of curves on surfaces
and the Riemannian Gaussian curvature of surfaces in the affine group.
In Section 4, we prove the Gauss-Bonnet theorem in the affine group. In Section 5, we compute the sub-Riemannian limit of curvature of curves in the group of rigid motions of the Minkowski plane. In Section 6, we compute sub-Riemannian limits of geodesic curvature of curves on surfaces and the Riemannian Gaussian curvature of surfaces in the group of rigid motions of the Minkowski plane and a Gauss-Bonnet theorem in the group of rigid motions of the Minkowski plane is also obtained.


\vskip 1 true cm

\section{ The sub-Riemannian limit of curvature of curves in the affine group}

Firstly we introduce some notations on the affine group. Let $\mathbb{G} $ be the affine group $(0,\infty)\times \mathbb{R}^2$ where the non-commutative
group law is given by
$$(a,b,c)\star (x,y,z)=(ax,ay+b,z+c).$$
Then $(1,0,0)$ is a unit element. Let
\begin{equation}
X_1=x_1\partial_{x_1}, ~~X_2=x_1\partial_{x_2}+\partial_{x_3},~~X_3=x_1\partial_{x_2}.
\end{equation}
Then
\begin{equation}
\partial_{x_1}=\frac{1}{x_1}X_1, ~~\partial_{x_2}=\frac{1}{x_1}X_3,~~\partial_{x_3}=X_2-X_3,
\end{equation}
and ${\rm span}\{X_1,X_2,X_3\}=T\mathbb{G}.$ Let $H={\rm span}\{X_1,X_2\}$ be the horizontal distribution on $\mathbb{G}$
Let $\omega_1=\frac{1}{x_1}dx_1,~~\omega_2=dx_3,~~\omega=\frac{1}{x_1}dx_2-dx_3.$ Then $H={\rm Ker}\omega$. For the constant $L>0$, let
$g_L=\omega_1\otimes \omega_1+\omega_2\otimes \omega_2+L\omega\otimes \omega,~~g=g_1$ be the Riemannian metric on $\mathbb{G}$. Then $X_1,X_2,\widetilde{X_3}:=L^{-\frac{1}{2}}X_3$ are orthonormal basis on $T\mathbb{G}$ with respect to $g_L$. We have
\begin{equation}
[X_1,X_2]=X_3,~~[X_2,X_3]=0,~~[X_1,X_3]=X_3.
\end{equation}
Let $\nabla^L$ be the Levi-Civita connection on $\mathbb{G}$ with respect to $g_L$. By the Koszul formula, we have
\begin{equation}
2\langle \nabla^L_{X_i}X_j,X_k\rangle_L=\langle[X_i,X_j],X_k\rangle_L-\langle[X_j,X_k],X_i\rangle_L+\langle[X_k,X_i],X_j\rangle_L,
\end{equation}
where $i,j,k=1,2,3$. By (2.3) and (2.4), we have
\vskip 0.5 true cm
\begin{lem}
Let $\mathbb{G}$ be the affine group, then
\begin{align}
&\nabla^L_{X_j}X_j=0,~~~1\leq j\leq 2,~~~\nabla^L_{X_1}X_2=\frac{1}{2}X_3,~~~ \nabla^L_{X_2}X_1=-\frac{1}{2}X_3,\\
&\nabla^L_{X_1}X_3=-\frac{L}{2}X_2,~~\nabla^L_{X_3}X_1=-\frac{L}{2}X_2-X_3,\notag\\
& \nabla^L_{X_2}X_3=\nabla^L_{X_3}X_2=\frac{L}{2}X_1,~~\nabla^L_{X_3}X_3=LX_1.\notag
\end{align}
\end{lem}
\vskip 0.5 true cm
\begin{defn}
Let $\gamma:[a,b]\rightarrow (\mathbb{G},g_L)$ be a Euclidean $C^1$-smooth curve. We say that $\gamma$ is regular if $\dot{\gamma}\neq 0$ for every $t\in [a,b].$ Moreover we say that
$\gamma(t)$ is a horizontal point of $\gamma$ if
$$\omega(\dot{\gamma}(t))=\frac{\dot{\gamma}_2(t)}{\gamma_1(t)}-\dot{\gamma}_3(t)=0.$$
\end{defn}
\vskip 0.5 true cm
\begin{defn}
Let $\gamma:[a,b]\rightarrow (\mathbb{G},g_L)$ be a Euclidean $C^2$-smooth regular curve in the Riemannian manifold $(\mathbb{G},g_L)$. The curvature $k^L_{\gamma}$ of $\gamma$ at $\gamma(t)$ is defined as
\begin{equation}
k^L_{\gamma}:=\sqrt{\frac{||\nabla^L_{\dot{\gamma}}{\dot{\gamma}}||_L^2}{||\dot{\gamma}||^4_L}-\frac{\langle \nabla^L_{\dot{\gamma}}{\dot{\gamma}},\dot{\gamma}\rangle^2_L}{||\dot{\gamma}||^6_L}}.
\end{equation}
\end{defn}
\vskip 0.5 true cm
\begin{lem}
Let $\gamma:[a,b]\rightarrow (\mathbb{G},g_L)$ be a Euclidean $C^2$-smooth regular curve in the Riemannian manifold $(\mathbb{G},g_L)$. Then
\begin{align}
k^L_{\gamma}&=\left\{\left\{\left[\frac{\ddot{\gamma}_1\gamma_1-(\dot{\gamma}_1)^2}{\gamma^2_1}+L\omega(\dot{\gamma}(t))\frac{\dot{\gamma}_2}{\gamma_1}\right]^2+\left[
\ddot{\gamma}_3-L\omega(\dot{\gamma}(t))\frac{\dot{\gamma}_1}{\gamma_1}\right]^2\right.\right.\\\notag
&\left.+L\left[\frac{d}{dt}(\omega(\dot{\gamma}(t)))-\omega(\dot{\gamma}(t))\frac{\dot{\gamma}_1}{\gamma_1}\right]^2\right\}
\cdot\left[\left(\frac{\dot{\gamma}_1}{\gamma_1}\right)^2+\dot{\gamma}_3^2+L(\omega(\dot{\gamma}(t)))^2\right]^{-2}\\\notag
&-\left\{\frac{\dot{\gamma}_1}{\gamma_1}\left[\frac{\ddot{\gamma}_1\gamma_1-(\dot{\gamma}_1)^2}{\gamma^2_1}+L\omega(\dot{\gamma}(t))\frac{\dot{\gamma}_2}{\gamma_1}\right]
+\dot{\gamma}_3\left[
\ddot{\gamma}_3-L\omega(\dot{\gamma}(t))\frac{\dot{\gamma}_1}{\gamma_1}\right]\right.\\\notag
&\left.\left.+L\omega(\dot{\gamma}(t))\left[\frac{d}{dt}(\omega(\dot{\gamma}(t)))-\omega(\dot{\gamma}(t))\frac{\dot{\gamma}_1}{\gamma_1}\right]\right\}^2
\cdot\left[\left(\frac{\dot{\gamma}_1}{\gamma_1}\right)^2+\dot{\gamma}_3^2+L(\omega(\dot{\gamma}(t)))^2\right]^{-3}\right\}^{\frac{1}{2}}.\\\notag
\end{align}
In particular, if $\gamma(t)$ is a horizontal point of $\gamma$,
\begin{align}
k^L_{\gamma}&=\left\{\left\{\left[\frac{\ddot{\gamma}_1\gamma_1-(\dot{\gamma}_1)^2}{\gamma^2_1}\right]^2+(
\ddot{\gamma}_3)^2+L\left[\frac{d}{dt}(\omega(\dot{\gamma}(t)))\right]^2\right\}
\cdot\left[\left(\frac{\dot{\gamma}_1}{\gamma_1}\right)^2+\dot{\gamma}_3^2\right]^{-2}\right.\\\notag
&\left.-\left\{\left[\frac{\ddot{\gamma}_1\dot{\gamma}_1\gamma_1-(\dot{\gamma}_1)^3}{\gamma^3_1}\right]
+\dot{\gamma}_3
\ddot{\gamma}_3\right\}^2
\cdot\left[\left(\frac{\dot{\gamma}_1}{\gamma_1}\right)^2+\dot{\gamma}_3^2\right]^{-3}\right\}^{\frac{1}{2}}.\\\notag
\end{align}
\end{lem}
\begin{proof}
By (2.2), we have
\begin{equation}
\dot{\gamma}(t)=\frac{\dot{\gamma}_1}{\gamma_1}X_1+\dot{\gamma}_3X_2+\omega(\dot{\gamma}(t))X_3.
\end{equation}
By Lemma 2.1 and (2.9), we have
\begin{align}
&\nabla^L_{\dot{\gamma}}X_1=-\frac{L}{2}\left(\frac{\dot{\gamma}_2}{\gamma_1}-\dot{\gamma}_3\right)X_2+\left(\frac{\dot{\gamma}_3}{2}-\frac{\dot{\gamma}_2}{\gamma_1}\right)X_3,\\\notag
&\nabla^L_{\dot{\gamma}}X_2=\frac{L}{2}\left(\frac{\dot{\gamma}_2}{\gamma_1}-\dot{\gamma}_3\right)X_1+\frac{1}{2}\frac{\dot{\gamma}_1}{\gamma_1}X_3,\\ \notag
&\nabla^L_{\dot{\gamma}}X_3=\left(-\frac{\dot{\gamma}_3}{2}+\frac{\dot{\gamma}_2}{\gamma_1}\right)LX_1-\frac{L}{2}\frac{\dot{\gamma}_1}{\gamma_1}X_2\notag
\end{align}
By (2.9) and (2.10), we have
\begin{align}
\nabla^L_{\dot{\gamma}}\dot{\gamma}&=
\left[\frac{\ddot{\gamma}_1\gamma_1-(\dot{\gamma}_1)^2}{\gamma^2_1}+L\omega(\dot{\gamma}(t))\frac{\dot{\gamma}_2}{\gamma_1}\right]X_1+\left[
\ddot{\gamma}_3-L\omega(\dot{\gamma}(t))\frac{\dot{\gamma}_1}{\gamma_1}\right]X_2\\\notag
&+\left[\frac{d}{dt}(\omega(\dot{\gamma}(t)))-\omega(\dot{\gamma}(t))\frac{\dot{\gamma}_1}{\gamma_1}\right]X_3\notag
\end{align}
By (2.6),(2.9) and (2.11), we get Lemma 2.4.
\end{proof}
\vskip 0.5 true cm
\begin{defn}
Let $\gamma:[a,b]\rightarrow (\mathbb{G},g_L)$ be a Euclidean $C^2$-smooth regular curve in the Riemannian manifold $(\mathbb{G},g_L)$.
We define the intrinsic curvature $k_{\gamma}^{\infty}$ of $\gamma$ at $\gamma(t)$ to be
$$k_{\gamma}^{\infty}:={\rm lim}_{L\rightarrow +\infty}k_{\gamma}^L,$$
if the limit exists.
\end{defn}
\indent We introduce the following notation: for continuous functions $f_1,f_2:(0,+\infty)\rightarrow \mathbb{R}$,
\begin{equation}
f_1(L)\sim f_2(L),~~as ~~L\rightarrow +\infty\Leftrightarrow {\rm lim}_{L\rightarrow +\infty}\frac{f_1(L)}{f_2(L)}=1.
\end{equation}
\vskip 0.5 true cm
\begin{lem}
Let $\gamma:[a,b]\rightarrow (\mathbb{G},g_L)$ be a Euclidean $C^2$-smooth regular curve in the Riemannian manifold $(\mathbb{G},g_L)$. Then
\begin{equation}
k_{\gamma}^{\infty}=\frac{\sqrt{\dot{\gamma}_1^2+\dot{\gamma}_2^2}}{|\gamma_1||\omega(\dot{\gamma}(t))|},~~if ~~\omega(\dot{\gamma}(t))\neq 0,
\end{equation}
\begin{align}
k^{\infty}_{\gamma}&=\left\{\left\{\left[\frac{\ddot{\gamma}_1\gamma_1-(\dot{\gamma}_1)^2}{\gamma^2_1}\right]^2+(
\ddot{\gamma}_3)^2\right\}
\cdot\left[\left(\frac{\dot{\gamma}_1}{\gamma_1}\right)^2+\dot{\gamma}_3^2\right]^{-2}\right.\\\notag
&\left.-\left\{\left[\frac{\ddot{\gamma}_1\dot{\gamma}_1\gamma_1-(\dot{\gamma}_1)^3}{\gamma^3_1}\right]
+\dot{\gamma}_3
\ddot{\gamma}_3\right\}^2
\cdot\left[\left(\frac{\dot{\gamma}_1}{\gamma_1}\right)^2+\dot{\gamma}_3^2\right]^{-3}\right\}^{\frac{1}{2}}\\\notag
&~~if ~~\omega(\dot{\gamma}(t))= 0 ~~and~~\frac{d}{dt}(\omega(\dot{\gamma}(t)))=0,\\\notag
\end{align}
\begin{equation}
{\rm lim}_{L\rightarrow +\infty}\frac{k_{\gamma}^{L}}{\sqrt{L}}=\frac{|\frac{d}{dt}(\omega(\dot{\gamma}(t)))|}{\left(\frac{\dot{\gamma}_1}{\gamma_1}\right)^2+\dot{\gamma}_3^2},~~if ~~\omega(\dot{\gamma}(t))= 0
~~and~~\frac{d}{dt}(\omega(\dot{\gamma}(t)))\neq 0.
\end{equation}
\end{lem}
\begin{proof}
Using the notation introduced in (2.12), when $\omega(\dot{\gamma}(t))\neq 0$, we have
$$||\nabla^L_{\dot{\gamma}}{\dot{\gamma}}||_L^2\sim (\frac{\omega(\dot{\gamma}(t))}{\gamma_1})^2(\dot{\gamma}_1^2+\dot{\gamma}_2^2)L^2,~~as~~L\rightarrow +\infty,$$
$$||\dot{\gamma}||^2_L\sim L\omega(\dot{\gamma}(t))^2,~~as~~L\rightarrow +\infty,$$
$$\langle \nabla^L_{\dot{\gamma}}{\dot{\gamma}},\dot{\gamma}\rangle^2_L\sim O(L^2)~~as~~L\rightarrow +\infty.$$
Therefore
$$\frac{||\nabla^L_{\dot{\gamma}}{\dot{\gamma}}||_L^2}{||\dot{\gamma}||^4_L}\rightarrow \frac{\dot{\gamma}_1^2+\dot{\gamma}_2^2}{{\gamma_1}^2\omega(\dot{\gamma}(t))^2},~~as~~L\rightarrow +\infty,$$
$$\frac{\langle \nabla^L_{\dot{\gamma}}{\dot{\gamma}},\dot{\gamma}\rangle^2_L}{||\dot{\gamma}||^6_L}\rightarrow 0,~~as~~L\rightarrow +\infty.$$
So by (2.6), we have (2.13). (2.14) comes from (2.8) and $\frac{d}{dt}(\omega(\dot{\gamma}(t)))=0$. When
$\omega(\dot{\gamma}(t))= 0$
~~and~~$\frac{d}{dt}(\omega(\dot{\gamma}(t)))\neq 0$, we have
$$||\nabla^L_{\dot{\gamma}}{\dot{\gamma}}||_L^2\sim L[\frac{d}{dt}(\omega(\dot{\gamma}(t)))]^2,~~as~~L\rightarrow +\infty,$$
$$||\dot{\gamma}||^2_L=\left(\frac{\dot{\gamma}_1}{\gamma_1}\right)^2+\dot{\gamma}_3^2,$$
$$\langle \nabla^L_{\dot{\gamma}}{\dot{\gamma}},\dot{\gamma}\rangle^2_L=O(1)~~as~~L\rightarrow +\infty.$$
By (2.6), we get (2.15).
\end{proof}

\section{The sub-Riemannian limit of geodesic curvature of curves on surfaces in the affine group}
\indent We will say that a surface $\Sigma\subset(\mathbb{G},g_L)$ is regular if $\Sigma$ is a Euclidean $C^2$-smooth compact and oriented surface. In particular we will assume that there exists
a Euclidean $C^2$-smooth function $u:\mathbb{G}\rightarrow \mathbb{R}$ such that
$$\Sigma=\{(x_1,x_2,x_3)\in \mathbb{G}:u(x_1,x_2,x_3)=0\}$$
and $u_{x_1}\partial_{x_1}+u_{x_2}\partial_{x_2}+u_{x_3}\partial_{x_3}\neq 0.$ Let $\nabla_Hu=X_1(u)X_1+X_2(u)X_2.$
 A point $x\in\Sigma$ is called {\it characteristic} if $\nabla_Hu(x)=0$.
 We define the characteristic set $C(\Sigma):=\{x\in\Sigma|\nabla_Hu(x)=0\}.$
  Our computations will
be local and away from characteristic points of $\Sigma$. Let us define first
$$p:=X_1u,~~~~q:=X_2u ,~~{\rm and}~~r:=\widetilde{X}_3u.$$
We then define
\begin{align}
&l:=\sqrt{p^2+q^2},~~~~l_L:=\sqrt{p^2+q^2+r^2},~~~~\overline{p}:=\frac{p}{l},\\
&\overline{q}:=\frac{q}{l},~~~~
\overline{p_L}:=\frac{p}{l_L},~~~~\overline{q_L}:=\frac{q}{l_L},~~~~\overline{r_L}:=\frac{r}{l_L}.\notag
\end{align}
In particular, $\overline{p}^2+\overline{q}^2=1$. These functions are well defined at every non-characteristic point. Let
\begin{align}
v_L=\overline{p_L}X_1+\overline{q_L}X_2+\overline{r_L}\widetilde{X_3},~~~~e_1=\overline{q}X_1-\overline{p}X_2,~~~~
e_2=\overline{r_L}~~\overline{p}X_1+\overline{r_L}~~ \overline{q}X_2-\frac{l}{l_L}\widetilde{X_3},
\end{align}
then $v_L$ is the Riemannian unit normal vector to $\Sigma$ and $e_1,e_2$ are the orthonormal basis of $\Sigma$. On $T\Sigma$ we define a linear transformation $J_L:T\Sigma\rightarrow T\Sigma$ such that
\begin{equation}
J_L(e_1):=e_2;~~~~J_L(e_2):=-e_1.
\end{equation}
For every $U,V\in T\Sigma$, we define $\nabla^{\Sigma,L}_UV=\pi \nabla^{L}_UV$ where $\pi:T\mathbb{G}\rightarrow T\Sigma$ is the projection. Then $\nabla^{\Sigma,L}$ is the Levi-Civita connection on $\Sigma$
with respect to the metric $g_L$. By (2.11),(3.2) and
\begin{equation}
\nabla^{\Sigma,L}_{\dot{\gamma}}\dot{\gamma}=\langle \nabla^L_{\dot{\gamma}}\dot{\gamma},e_1\rangle_Le_1+\langle \nabla^L_{\dot{\gamma}}\dot{\gamma},e_2\rangle_Le_2,
\end{equation}
we have
\begin{align}
\nabla^{\Sigma,L}_{\dot{\gamma}}\dot{\gamma}&=
\left\{\overline{q}\left[\frac{\ddot{\gamma}_1\gamma_1-(\dot{\gamma}_1)^2}{\gamma^2_1}+L\omega(\dot{\gamma}(t))\frac{\dot{\gamma}_2}{\gamma_1}\right]-\overline{p}\left[
\ddot{\gamma}_3-L\omega(\dot{\gamma}(t))\frac{\dot{\gamma}_1}{\gamma_1}\right]\right\}e_1\\\notag
&+\left\{\overline{r_L}~~\overline{p}\left[\frac{\ddot{\gamma}_1\gamma_1-(\dot{\gamma}_1)^2}{\gamma^2_1}+L\omega(\dot{\gamma}(t))\frac{\dot{\gamma}_2}{\gamma_1}\right]+\overline{r_L}~~ \overline{q}\left[
\ddot{\gamma}_3-L\omega(\dot{\gamma}(t))\frac{\dot{\gamma}_1}{\gamma_1}\right]\right.\\
&\left.-\frac{l}{l_L}L^{\frac{1}{2}}\left[\frac{d}{dt}(\omega(\dot{\gamma}(t)))-\omega(\dot{\gamma}(t))\frac{\dot{\gamma}_1}{\gamma_1}\right]\right\}e_2.\notag
\end{align}
Moreover if $\omega(\dot{\gamma}(t))=0$, then
\begin{align}
\nabla^{\Sigma,L}_{\dot{\gamma}}\dot{\gamma}&=
\left\{\overline{q}\left[\frac{\ddot{\gamma}_1\gamma_1-(\dot{\gamma}_1)^2}{\gamma^2_1}\right]-\overline{p}
\ddot{\gamma}_3\right\}e_1\\\notag
&+\left\{\overline{r_L}~~\overline{p}\left[\frac{\ddot{\gamma}_1\gamma_1-(\dot{\gamma}_1)^2}{\gamma^2_1}\right]+\overline{r_L}~~ \overline{q}
\ddot{\gamma}_3-\frac{l}{l_L}L^{\frac{1}{2}}\frac{d}{dt}(\omega(\dot{\gamma}(t)))\right\}e_2.\notag
\end{align}
\vskip 0.5 true cm
\begin{defn}
Let $\Sigma\subset(\mathbb{G},g_L)$ be a regular surface.
Let $\gamma:[a,b]\rightarrow \Sigma$ be a Euclidean $C^2$-smooth regular curve. The geodesic curvature $k^L_{\gamma,\Sigma}$ of $\gamma$ at $\gamma(t)$ is defined as
\begin{equation}
k^L_{\gamma,\Sigma}:=\sqrt{\frac{||\nabla^{\Sigma,L}_{\dot{\gamma}}{\dot{\gamma}}||_{\Sigma,L}^2}{||\dot{\gamma}||^4_{\Sigma,L}}-\frac{\langle \nabla^{\Sigma,L}_{\dot{\gamma}}{\dot{\gamma}},\dot{\gamma}\rangle^2_{\Sigma,L}}{||\dot{\gamma}||^6_{\Sigma,L}}}.
\end{equation}
\end{defn}
\vskip 0.5 true cm
\begin{defn}
Let $\Sigma\subset(\mathbb{G},g_L)$ be a regular surface. Let $\gamma:[a,b]\rightarrow \Sigma$ be a Euclidean $C^2$-smooth regular curve.
We define the intrinsic geodesic curvature $k_{\gamma,\Sigma}^{\infty}$ of $\gamma$ at $\gamma(t)$ to be
$$k_{\gamma,\Sigma}^{\infty}:={\rm lim}_{L\rightarrow +\infty}k_{\gamma,\Sigma}^L,$$
if the limit exists.
\end{defn}
\vskip 0.5 true cm
\begin{lem}
Let $\Sigma\subset(\mathbb{G},g_L)$ be a regular surface.
Let $\gamma:[a,b]\rightarrow \Sigma$ be a Euclidean $C^2$-smooth regular curve. Then
\begin{equation}
k_{\gamma,\Sigma}^{\infty}=\frac{|\overline{p}\dot{\gamma}_1+\overline{q}\dot{\gamma}_2|}{|\gamma_1||\omega(\dot{\gamma}(t))|},~~if ~~\omega(\dot{\gamma}(t))\neq 0,
\end{equation}
$$k^{\infty}_{\gamma,\Sigma}=0
~~if ~~\omega(\dot{\gamma}(t))= 0, ~~and~~\frac{d}{dt}(\omega(\dot{\gamma}(t)))=0,$$
\begin{equation}
{\rm lim}_{L\rightarrow +\infty}\frac{k_{\gamma,\Sigma}^{L}}{\sqrt{L}}=\frac{|\frac{d}{dt}(\omega(\dot{\gamma}(t)))|}{\left(\overline{q}\frac{\dot{\gamma}_1}{\gamma_1}-\overline{p}\dot{\gamma}_3\right)^2},~~if ~~\omega(\dot{\gamma}(t))= 0
~~and~~\frac{d}{dt}(\omega(\dot{\gamma}(t)))\neq 0.
\end{equation}
\end{lem}

\begin{proof} By (2.9) and $\dot{\gamma}\in T\Sigma$, we have
\begin{equation}
\dot{\gamma}=(\overline{q}\frac{\dot{\gamma}_1}{\gamma_1}-\overline{p}\dot{\gamma}_3)e_1-\frac{l_L}{l}L^{\frac{1}{2}}\omega(\dot{\gamma}(t))e_2.
\end{equation}
By (3.6), we have
\begin{align}
||\nabla^{\Sigma,L}_{\dot{\gamma}}\dot{\gamma}||^2_{L,\Sigma}&=
\left\{\overline{q}\left[\frac{\ddot{\gamma}_1\gamma_1-(\dot{\gamma}_1)^2}{\gamma^2_1}+L\omega(\dot{\gamma}(t))\frac{\dot{\gamma}_2}{\gamma_1}\right]-\overline{p}\left[
\ddot{\gamma}_3-L\omega(\dot{\gamma}(t))\frac{\dot{\gamma}_1}{\gamma_1}\right]\right\}^2\\\notag
&+\left\{\overline{r_L}~~\overline{p}\left[\frac{\ddot{\gamma}_1\gamma_1-(\dot{\gamma}_1)^2}{\gamma^2_1}+L\omega(\dot{\gamma}(t))\frac{\dot{\gamma}_2}{\gamma_1}\right]+\overline{r_L}~~ \overline{q}\left[
\ddot{\gamma}_3-L\omega(\dot{\gamma}(t))\frac{\dot{\gamma}_1}{\gamma_1}\right]\right.\\\notag
&\left.-\frac{l}{l_L}L^{\frac{1}{2}}\left[\frac{d}{dt}(\omega(\dot{\gamma}(t)))-\omega(\dot{\gamma}(t))\frac{\dot{\gamma}_1}{\gamma_1}\right]\right\}^2\\\notag
&\sim L^2\frac{(\overline{p}\dot{\gamma}_1+\overline{q}\dot{\gamma}_2)^2\omega(\dot{\gamma}(t))^2}{\gamma_1^2},~~~~ {\rm as} ~~~~L\rightarrow +\infty.\notag
\end{align}
Similarly, we have that when $\omega(\dot{\gamma}(t))\neq 0$,
\begin{equation}
||\dot{\gamma}||_{\Sigma,L}=\sqrt{(\overline{q}\frac{\dot{\gamma}_1}{\gamma_1}-\overline{p}\dot{\gamma}_3)^2+(\frac{l_L}{l})^2L\omega(\dot{\gamma}(t))^2}\sim L^{\frac{1}{2}}|\omega(\dot{\gamma}(t))|
,~~~~ {\rm as} ~~~~L\rightarrow +\infty.
\end{equation}
By (3.6) and (3.10), we have

\begin{align}
\langle \nabla^{\Sigma,L}_{\dot{\gamma}}{\dot{\gamma}},\dot{\gamma}\rangle_{\Sigma,L}&=
(\overline{q}\frac{\dot{\gamma}_1}{\gamma_1}
-\overline{p}\dot{\gamma}_3)
\cdot\left\{\overline{q}\left[\frac{\ddot{\gamma}_1\gamma_1-(\dot{\gamma}_1)^2}{\gamma^2_1}+L\omega(\dot{\gamma}(t))\frac{\dot{\gamma}_2}{\gamma_1}\right]
-\overline{p}\left[
\ddot{\gamma}_3-L\omega(\dot{\gamma}(t))\frac{\dot{\gamma}_1}{\gamma_1}\right]\right\}\\\notag
&-\frac{l_L}{l}L^{\frac{1}{2}}\omega(\dot{\gamma}(t))
\left\{\overline{r_L}~~\overline{p}\left[\frac{\ddot{\gamma}_1\gamma_1-(\dot{\gamma}_1)^2}{\gamma^2_1}+L\omega(\dot{\gamma}(t))\frac{\dot{\gamma}_2}{\gamma_1}\right]+\overline{r_L}~~ \overline{q}\left[
\ddot{\gamma}_3-L\omega(\dot{\gamma}(t))\frac{\dot{\gamma}_1}{\gamma_1}\right]\right.\\\notag
&\left.-\frac{l}{l_L}L^{\frac{1}{2}}\left[\frac{d}{dt}(\omega(\dot{\gamma}(t)))-\omega(\dot{\gamma}(t))\frac{\dot{\gamma}_1}{\gamma_1}\right]\right\}\notag
\sim M_0L,
\end{align}
where $M_0$ does not depend on $L$.
By (3.7),(3.11)-(3.13), we get (3.8).
When $\omega(\dot{\gamma}(t))= 0$ and $\frac{d}{dt}(\omega(\dot{\gamma}(t)))=0$, we have
\begin{align}
||\nabla^{\Sigma,L}_{\dot{\gamma}}\dot{\gamma}||^2_{L,\Sigma}&=
\left[\overline{q}\frac{\ddot{\gamma}_1\gamma_1-(\dot{\gamma}_1)^2}{\gamma^2_1}-\overline{p}
\ddot{\gamma}_3\right]^2
+\left[\overline{r_L}~~\overline{p}\frac{\ddot{\gamma}_1\gamma_1-(\dot{\gamma}_1)^2}{\gamma^2_1}+\overline{r_L}~~ \overline{q}
\ddot{\gamma}_3
\right]^2\\\notag
&\sim \left[\overline{q}\frac{\ddot{\gamma}_1\gamma_1-(\dot{\gamma}_1)^2}{\gamma^2_1}-\overline{p}
\ddot{\gamma}_3\right]^2,~~~~ {\rm as} ~~~~L\rightarrow +\infty,\notag
\end{align}
and
\begin{equation}
||\dot{\gamma}||_{\Sigma,L}=|\overline{q}\frac{\dot{\gamma}_1}{\gamma_1}-\overline{p}\dot{\gamma}_3|,
\end{equation}
\begin{align}
\langle \nabla^{\Sigma,L}_{\dot{\gamma}}{\dot{\gamma}},\dot{\gamma}\rangle_{\Sigma,L}&=
(\overline{q}\frac{\dot{\gamma}_1}{\gamma_1}
-\overline{p}\dot{\gamma}_3)
\cdot\left(\overline{q}\frac{\ddot{\gamma}_1\gamma_1-(\dot{\gamma}_1)^2}{\gamma^2_1}
-\overline{p}
\ddot{\gamma}_3\right).
\end{align}
By (3.14)-(3.16) and (3.7), we get $k^{\infty}_{\gamma,\Sigma}=0$.
When $\omega(\dot{\gamma}(t))= 0$ and $\frac{d}{dt}(\omega(\dot{\gamma}(t)))\neq 0$, we have
$$||\nabla^{\Sigma,L}_{\dot{\gamma}}\dot{\gamma}||^2_{L,\Sigma}\sim L
[\frac{d}{dt}(\omega(\dot{\gamma}(t)))]^2,$$
$$\langle \nabla^{\Sigma,L}_{\dot{\gamma}}{\dot{\gamma}},\dot{\gamma}\rangle_{\Sigma,L}=O(1),$$
so we get (3.9).
\end{proof}
\vskip 0.5 true cm
\begin{defn}
Let $\Sigma\subset(\mathbb{G},g_L)$ be a regular surface.
Let $\gamma:[a,b]\rightarrow \Sigma$ be a Euclidean $C^2$-smooth regular curve. The signed geodesic curvature $k^{L,s}_{\gamma,\Sigma}$ of $\gamma$ at $\gamma(t)$ is defined as
\begin{equation}
k^{L,s}_{\gamma,\Sigma}:=\frac{\langle \nabla^{\Sigma,L}_{\dot{\gamma}}{\dot{\gamma}},J_L(\dot{\gamma})\rangle_{\Sigma,L}}{||\dot{\gamma}||^3_{\Sigma,L}},
\end{equation}
where $J_L$ is defined by (3.3).
\end{defn}
\vskip 0.5 true cm
\begin{defn}
Let $\Sigma\subset(\mathbb{G},g_L)$ be a regular surface. Let $\gamma:[a,b]\rightarrow \Sigma$ be a Euclidean $C^2$-smooth regular curve.
We define the intrinsic geodesic curvature $k_{\gamma,\Sigma}^{\infty}$ of $\gamma$ at the non-characteristic point $\gamma(t)$ to be
$$k_{\gamma,\Sigma}^{\infty,s}:={\rm lim}_{L\rightarrow +\infty}k_{\gamma,\Sigma}^{L,s},$$
if the limit exists.
\end{defn}
\vskip 0.5 true cm

\begin{lem}
Let $\Sigma\subset(\mathbb{G},g_L)$ be a regular surface.
Let $\gamma:[a,b]\rightarrow \Sigma$ be a Euclidean $C^2$-smooth regular curve. Then
\begin{equation}
k_{\gamma,\Sigma}^{\infty,s}=\frac{\overline{p}\dot{\gamma}_1+\overline{q}\dot{\gamma}_2}{\gamma_1|\omega(\dot{\gamma}(t))|},~~if ~~\omega(\dot{\gamma}(t))\neq 0,
\end{equation}
$$k^{\infty,s}_{\gamma,\Sigma}=0
~~if ~~\omega(\dot{\gamma}(t))= 0, ~~and~~\frac{d}{dt}(\omega(\dot{\gamma}(t)))=0,$$
\begin{equation}
{\rm lim}_{L\rightarrow +\infty}\frac{k_{\gamma,\Sigma}^{L,s}}{\sqrt{L}}=\frac{(-\overline{q}\frac{\dot{\gamma}_1}{\gamma_1}+\overline{p}\dot{\gamma}_3)\frac{d}{dt}(\omega(\dot{\gamma}(t)))}{|\overline{q}\frac{\dot{\gamma}_1}{\gamma_1}-\overline{p}\dot{\gamma}_3|^3},~~if ~~\omega(\dot{\gamma}(t))= 0
~~and~~\frac{d}{dt}(\omega(\dot{\gamma}(t)))\neq 0.
\end{equation}
\end{lem}
\begin{proof} By (3.3) and (3.10), we have
\begin{equation}
J_L(\dot{\gamma})=\frac{l_L}{l}L^{\frac{1}{2}}\omega(\dot{\gamma}(t))e_1+(\overline{q}\frac{\dot{\gamma}_1}{\gamma_1}-\overline{p}\dot{\gamma}_3)e_2.
\end{equation}
By (3.5) and (3.20), we have
\begin{align}
\langle \nabla^{\Sigma,L}_{\dot{\gamma}}\dot{\gamma},J_L(\dot{\gamma})\rangle_{L,\Sigma}&=
\frac{l_L}{l}L^{\frac{1}{2}}\omega(\dot{\gamma}(t))
\left\{\overline{q}\left[\frac{\ddot{\gamma}_1\gamma_1-(\dot{\gamma}_1)^2}{\gamma^2_1}+L\omega(\dot{\gamma}(t))\frac{\dot{\gamma}_2}{\gamma_1}\right]-\overline{p}\left[
\ddot{\gamma}_3-L\omega(\dot{\gamma}(t))\frac{\dot{\gamma}_1}{\gamma_1}\right]\right\}\\\notag
&+(\overline{q}\frac{\dot{\gamma}_1}{\gamma_1}-\overline{p}\dot{\gamma}_3)\cdot
\left\{\overline{r_L}~~\overline{p}\left[\frac{\ddot{\gamma}_1\gamma_1-(\dot{\gamma}_1)^2}{\gamma^2_1}+L\omega(\dot{\gamma}(t))\frac{\dot{\gamma}_2}{\gamma_1}\right]+\overline{r_L}~~ \overline{q}\left[
\ddot{\gamma}_3-L\omega(\dot{\gamma}(t))\frac{\dot{\gamma}_1}{\gamma_1}\right]\right.\\\notag
&\left.-\frac{l}{l_L}L^{\frac{1}{2}}\left[\frac{d}{dt}(\omega(\dot{\gamma}(t)))-\omega(\dot{\gamma}(t))\frac{\dot{\gamma}_1}{\gamma_1}\right]\right\},\\\notag
&\sim L^{\frac{3}{2}}\omega(\dot{\gamma}(t))^2\frac{\overline{p}\dot{\gamma}_1+\overline{q}\dot{\gamma}_2}{\gamma_1}~~{\rm as}~~L\rightarrow +\infty.\notag
\end{align}
So we get (3.18). When $\omega(\dot{\gamma}(t))= 0$ and $\frac{d}{dt}(\omega(\dot{\gamma}(t)))=0,$
we get
\begin{align}
\langle \nabla^{\Sigma,L}_{\dot{\gamma}}\dot{\gamma},J_L(\dot{\gamma})\rangle_{L,\Sigma}=
(\overline{q}\frac{\dot{\gamma}_1}{\gamma_1}-\overline{p}\dot{\gamma}_3)\cdot
\left[\overline{r_L}~~\overline{p}\frac{\ddot{\gamma}_1\gamma_1-(\dot{\gamma}_1)^2}{\gamma^2_1}+\overline{r_L}~~ \overline{q}
\ddot{\gamma}_3\right]
\sim M_0L^{-\frac{1}{2}}~~{\rm as}~~L\rightarrow +\infty.
\end{align}
So $k^{\infty,s}_{\gamma,\Sigma}=0.$ When $\omega(\dot{\gamma}(t))= 0$ and $\frac{d}{dt}(\omega(\dot{\gamma}(t)))\neq 0$, we have

\begin{align}
\langle \nabla^{\Sigma,L}_{\dot{\gamma}}\dot{\gamma},J_L(\dot{\gamma})\rangle_{L,\Sigma}\sim
L^{\frac{1}{2}}(-\overline{q}\frac{\dot{\gamma}_1}{\gamma_1}+\overline{p}\dot{\gamma}_3)\frac{d}{dt}(\omega(\dot{\gamma}(t)))~~{\rm as}~~L\rightarrow +\infty.
\end{align}
So we get (3.19).
\end{proof}

In the following, we compute the sub-Riemannian limit of the Riemannian Gaussian curvature of surfaces in the affine group. We define the {\it second fundamental form} $II^L$ of the
embedding of $\Sigma$ into $(\mathbb{G},g_L)$:
\begin{equation}
II^L=\left(
  \begin{array}{cc}
   \langle \nabla^{L}_{e_1}v_L,e_1)\rangle_{L},
    & \langle \nabla^{L}_{e_1}v_L,e_2)\rangle_{L} \\
   \langle \nabla^{L}_{e_2}v_L,e_1)\rangle_{L},
    & \langle \nabla^{L}_{e_2}v_L,e_2)\rangle_{L} \\
  \end{array}
\right).
\end{equation}
Similarly to Theorem 4.3 in \cite{CDPT}, we have
\vskip 0.5 true cm
\begin{thm} The second fundamental form $II^L$ of the
embedding of $\Sigma$ into $(\mathbb{G},g_L)$ is given by
\begin{equation}
II^L=\left(
  \begin{array}{cc}
   \frac{l}{l_L}[X_1(\overline{p})+X_2(\overline{q})],
    & -\frac{l_L}{l}\langle e_1,\nabla_H(\overline{r_L})\rangle_L-\frac{\sqrt{L}}{2} \\
  -\frac{l_L}{l}\langle e_1,\nabla_H(\overline{r_L})\rangle_L-\frac{\sqrt{L}}{2}  ,
    & -\frac{l^2}{l_L^2}\langle e_2,\nabla_H(\frac{r}{l})\rangle_L+\widetilde{X_3}(\overline{r_L})-\overline{p_L} \\
  \end{array}
\right).
\end{equation}
\end{thm}
\vskip 0.5 true cm
\indent The Riemannian mean curvature $\mathcal{H}_L$ of $\Sigma$ is defined by
$$\mathcal{H}_L:={\rm tr}(II^L).$$
Define the curvature of a connection $\nabla$ by
\begin{equation}
R(X,Y)Z=\nabla_X\nabla_Y-\nabla_Y\nabla_X-\nabla_{[X,Y]}.
\end{equation}
Let
\begin{equation}
\mathcal{K}^{\Sigma,L}(e_1,e_2)=-\langle R^{\Sigma,L}(e_1,e_2)e_1,e_2\rangle_{\Sigma,L},~~~~\mathcal{K}^{L}(e_1,e_2)=-\langle R^{L}(e_1,e_2)e_1,e_2\rangle_L.
\end{equation}
By the Gauss equation, we have
\begin{equation}
\mathcal{K}^{\Sigma,L}(e_1,e_2)=\mathcal{K}^{L}(e_1,e_2)+{\rm det}(II^L).
\end{equation}
\vskip 0.5 true cm
\begin{prop} Away from characteristic points, the horizontal mean curvature $\mathcal{H}_{\infty}$ of $\Sigma\subset\mathbb{G}$ is given by
\begin{equation}
\mathcal{H}_{\infty}={\rm lim}_{L\rightarrow +\infty}\mathbb{}\mathcal{H}_L=X_1(\overline{p})+X_2(\overline{q})-\overline{p}.
\end{equation}
\end{prop}
\begin{proof} By
$$\frac{l^2}{l_L^2}\langle e_2,\nabla_H(\frac{r}{l})\rangle_L=\frac{\overline{p}r}{l}X_1(\overline{r_L})+\frac{\overline{q}r}{l}X_2(\overline{r_L})=O(L^{-1})$$
$$\frac{l}{l_L}[X_1(\overline{p})+X_2(\overline{q})]\rightarrow X_1(\overline{p})+X_2(\overline{q}),~~~~\widetilde{X_3}(\overline{r_L})\rightarrow 0,~~~~\overline{p_L}\rightarrow\overline{p},$$
we get (3.29).
\end{proof}
\vskip 0.5 true cm
By Lemma 2.1 and (3.26), we have
\vskip 0.5 true cm
\begin{lem}
Let $\mathbb{G}$ be the affine group, then
\begin{align}
&
R^L(X_1,X_2)X_1=\frac{3}{4}LX_2+X_3,~~~ R^L(X_1,X_2)X_2=-\frac{3}{4}LX_1,
~~~ R^L(X_1,X_2)X_3=-LX_1,~~~\\\notag
&R^L(X_1,X_3)X_1=LX_2+\frac{3}{4}LX_3,~~~ R^L(X_1,X_3)X_2=-LX_1,
~~~ R^L(X_1,X_3)X_3=(\frac{L^2}{4}-L)X_1,~~~\\\notag
& R^L(X_2,X_3)X_1=0,~~~
R^L(X_2,X_3)X_2=-\frac{L}{4}X_3,~~~ R^L(X_2,X_3)X_3=\frac{L^2}{4}X_2.\notag
\end{align}
\end{lem}
\vskip 0.5 true cm
\begin{prop} Away from characteristic points, we have
\begin{equation}
\mathcal{K}^{\Sigma,L}(e_1,e_2)\rightarrow -\overline{q}^2L+A+O(\frac{1}{\sqrt{L}}),~~{\rm as}~~L\rightarrow +\infty,
\end{equation}
where
\begin{equation}
A:=-\langle e_1,\nabla_H(\frac{X_3u}{|\nabla_Hu|})\rangle-\overline{p}[X_1(\overline{p})+X_2(\overline{q})]
-\overline{p}^2\frac{(X_3u)^2}{l^2}+2\overline{q}\frac{X_3u}{l}.
\end{equation}
\end{prop}
\begin{proof} By (3.2), we have
\begin{align}
&~~~~\langle R^{L}(e_1,e_2)e_1,e_2\rangle_L\\\notag
&=\overline{r_L}^2\langle R^{L}(X_1,X_2)X_1,X_2\rangle_L-2\frac{l}{l_L}\overline{q}L^{-\frac{1}{2}}\overline{r_L}\langle R^{L}(X_1,X_2)X_1,X_3\rangle_L\\\notag
&+2\frac{l}{l_L}\overline{p}L^{-\frac{1}{2}}\overline{r_L}\langle R^{L}(X_1,X_2)X_2,X_3\rangle_L+(\frac{l}{l_L}\overline{q})^2L^{-1}\langle R^{L}(X_1,X_3)X_1,X_3\rangle_L\\\notag
&-2(\frac{l}{l_L})^2\overline{p}\overline{q}L^{-1}\langle R^{L}(X_1,X_3)X_2,X_3\rangle_L+(\overline{p}\frac{l}{l_L})^2L^{-1}\langle R^{L}(X_2,X_3)X_2,X_3\rangle_L.\\\notag
\end{align}
By Lemma 3.9, we have
\begin{align}
\mathcal{K}^{L}(e_1,e_2)=\left[\frac{1}{4}(\overline{p}\frac{l}{l_L})^2-\frac{3}{4}(\overline{q}\frac{l}{l_L})^2\right]L-\frac{3}{4}L\overline{r_L}^2+2\frac{l}{l_L}\overline{q}L^{\frac{1}{2}}\overline{r_L}.
\end{align}
By (3.25) and
$$\nabla_H(\overline{r_L})=L^{-\frac{1}{2}}\nabla_H(\frac{X_3u}{|\nabla_Hu|})+O(L^{-1})~~{\rm as}~~L\rightarrow +\infty$$
we get
\begin{equation}
{\rm det}(II^L)=-\frac{L}{4}-\langle e_1,\nabla_H(\frac{X_3u}{|\nabla_Hu|})\rangle-\overline{p}[X_1(\overline{p})+X_2(\overline{q})]
+O(L^{-\frac{1}{2}})~~{\rm as}~~L\rightarrow +\infty.
\end{equation}
By (3.28),(3.34),(3.35) and
\begin{equation}
\left[\frac{1}{4}(\overline{p}\frac{l}{l_L})^2-\frac{3}{4}(\overline{q}\frac{l}{l_L})^2-\frac{1}{4}\right]L
=-\overline{q}^2L-(\frac{1}{4}\overline{p}^2-\frac{3}{4}\overline{q}^2)\frac{(X_3u)^2}{l^2}+
O(L^{-\frac{1}{2}})~~{\rm as}~~L\rightarrow +\infty,
\end{equation}
we get (3.31).
\end{proof}

\section{A Gauss-Bonnet theorem in the affine group}

Let us first consider the case of a regular curve $\gamma:[a,b]\rightarrow (\mathbb{G},g_L)$. We define the Riemannian length measure
$ds_L=||\dot{\gamma}||_Ldt.$
\vskip 0.5 true cm
\begin{lem}
Let $\gamma:[a,b]\rightarrow (\mathbb{G},g_L)$ be a Euclidean $C^2$-smooth and regular curve. Let
\begin{equation}
d{s}:=|\omega(\dot{\gamma}(t))|dt,~~~~d\overline{s}:=\frac{1}{2}\frac{1}{|\omega(\dot{\gamma}(t))|}\left(\frac{\dot{\gamma}_1^2}{\gamma_1^2}+\dot{\gamma}_3^2\right)dt.
\end{equation}
Then
\begin{equation}
{\rm lim}_{L\rightarrow +\infty}\frac{1}{\sqrt{L}}\int_{\gamma}ds_L=\int_a^bds.
\end{equation}
When $\omega(\dot{\gamma}(t))\neq 0$, we have
\begin{equation}
\frac{1}{\sqrt{L}}ds_L=ds+d\overline{s}L^{-1}+O(L^{-2}) ~~{\rm as}~~L\rightarrow +\infty.
\end{equation}
When $\omega(\dot{\gamma}(t))= 0$, we have
\begin{equation}
\frac{1}{\sqrt{L}}ds_L= \frac{1}{\sqrt{L}}\sqrt{\frac{\dot{\gamma}_1^2}{\gamma_1^2}+\dot{\gamma}_3^2}dt.
\end{equation}
\end{lem}
\begin{proof}
We know that $||\dot{\gamma}(t)||_L=\sqrt{\left(\frac{\dot{\gamma}_1}{\gamma_1}\right)^2+\dot{\gamma}_3^2+L\omega(\dot{\gamma}(t))^2}$, similar to the proof of Lemma 6.1 in \cite{BTV}, we can
prove (4.2).
When $\omega(\dot{\gamma}(t))\neq 0$, we have
$$
\frac{1}{\sqrt{L}}ds_L=\sqrt{L^{-1}\left(\left(\frac{\dot{\gamma}_1}{\gamma_1}\right)^2+\dot{\gamma}_3^2\right)+\omega(\dot{\gamma}(t))^2}dt.$$
Using the Taylor expansion, we can prove (4.3). From the definition of $ds_L$ and $\omega(\dot{\gamma}(t))= 0$, we get (4.4).
\end{proof}
\vskip 0.5 true cm
\begin{prop}
Let $\Sigma\subset(\mathbb{G},g_L)$ be a Euclidean $C^2$-smooth surface and $\Sigma=\{u=0\}$. Let $d\sigma_{\Sigma,L}$ denote the surface measure on $\Sigma$ with respect to the Riemannian metric $g_L$. Let
\begin{equation}
d\sigma_\Sigma:=(\overline{p}\omega_2-\overline{q}\omega_1)\wedge \omega,~~~~d\overline{\sigma_\Sigma}:=\frac{X_3u}{l}\omega_1\wedge \omega_2-\frac{(X_3u)^2}{2l^2}(\overline{p}\omega_2-\overline{q}\omega_1)\wedge \omega.
\end{equation}
Then
\begin{equation}
\frac{1}{\sqrt{L}}d\sigma_{\Sigma,L}=d\sigma_\Sigma+d\overline{\sigma_\Sigma}L^{-1}+O(L^{-2}),~~{\rm as}~~L\rightarrow +\infty.
\end{equation}
If $\Sigma=f(D)$ with
$$f=f(u_1,u_2)=(f_1,f_2,f_3):D\subset\mathbb{R}^2\rightarrow \mathbb{G},$$
then
\begin{align}
&{\rm lim}_{L\rightarrow +\infty}\frac{1}{\sqrt{L}}\int_{\Sigma}d\sigma_{\Sigma,L}=\int_D\left\{\left[\frac{(f_3)_{u_1}(f_2)_{u_2}-(f_3)_{u_2}(f_2)_{u_1}}{f_1}-2(f_3)_{u_1}(f_3)_{u_2}\right]^2\right.\\
&\left.+\left[\frac{(f_1)_{u_1}(f_2)_{u_2}-(f_1)_{u_2}(f_2)_{u_1}}{f_1^2}+\frac{(f_1)_{u_2}(f_3)_{u_1}-(f_1)_{u_1}(f_3)_{u_2}}{f_1}\right]^2\right\}^{\frac{1}{2}}du_1du_2.\notag
\end{align}
\end{prop}
\begin{proof}
We know that $$g_L(X_1,\cdot)=\omega_1,~~~~ g_L(X_2,\cdot)=\omega_2,~~~~ g_L(X_3,\cdot)=L\omega.$$ We define
$e_1^{\star}:=g_L(e_1,\cdot),~~e_2^{\star}:=g_L(e_2,\cdot),$ then
\begin{equation}
e_1^{\star}=\overline{q}\omega_1-\overline{p}\omega_2,~~~~
e_2^{\star}=\overline{r_L}~\overline{p}\omega_1+\overline{r_L}~\overline{q}\omega_2-\frac{l}{l_L}L^{\frac{1}{2}}\omega.
\end{equation}
Then
\begin{equation}
\frac{1}{\sqrt{L}}d\sigma_{\Sigma,L}=\frac{1}{\sqrt{L}}e_1^{\star}\wedge e_2^{\star}=\frac{l}{l_L}(\overline{p}\omega_2-\overline{q}\omega_1)\wedge \omega+\frac{1}{\sqrt{L}}\overline{r_L}\omega_1\wedge \omega_2.
\end{equation}
By
$$\overline{r_L}=\frac{(X_3u)L^{-\frac{1}{2}}}{\sqrt{p^2+q^2+L^{-1}(X_3u)^2}}$$
and the Taylor expansion
$$\frac{1}{l_L}=\frac{1}{l}-\frac{1}{2l^3}(X_3u)^2L^{-1}+O(L^{-2})~~{\rm as}~~L\rightarrow +\infty$$
we get (4.6). By (2.2), we have
\begin{equation}
f_{u_1}=(f_1)_{u_1}\partial_{x_1}+(f_2)_{u_1}\partial_{x_2}+(f_3)_{u_1}\partial_{x_3}=\frac{(f_1)_{u_1}}{f_1}X_1+(f_3)_{u_1}X_2+\sqrt{L}\left[\frac{(f_2)_{u_1}}{f_1}-(f_3)_{u_1}\right]\widetilde{X_3},
\end{equation}
and
\begin{equation}
f_{u_2}=\frac{(f_1)_{u_2}}{f_1}X_1+(f_3)_{u_2}X_2+\sqrt{L}\left[\frac{(f_2)_{u_2}}{f_1}-(f_3)_{u_2}\right]\widetilde{X_3}.
\end{equation}
Let
\begin{equation}
\overline{v_L}=\left|
  \begin{array}{ccc}
  X_1,
    & X_2,&\widetilde{X_3} \\
  \frac{(f_1)_{u_1}}{f_1} ,
    & (f_3)_{u_1}, & \sqrt{L}\left[\frac{(f_2)_{u_1}}{f_1}-(f_3)_{u_1}\right] \\
  \frac{(f_1)_{u_2}}{f_1}, &  (f_3)_{u_2}, & \sqrt{L}\left[\frac{(f_2)_{u_2}}{f_1}-(f_3)_{u_2}\right]\\
  \end{array}
\right|.
\end{equation}
We know that
$$d\sigma_{\Sigma,L}=\sqrt{{\rm det}(g_{ij})}du_1du_2,~~~~g_{ij}=g_L(f_{u_i},f_{u_j}),~~~~{\rm det}(g_{ij})=||\overline{v_L}||^2_L,$$
so by the dominated convergence theorem, we get (4.7).
\end{proof}

\vskip 0.5 true cm
\begin{thm}
 Let $\Sigma\subset (\mathbb{G},g_L)$
  be a regular surface with finitely many boundary components $(\partial\Sigma)_i,$ $i\in\{1,\cdots,n\}$, given by Euclidean $C^2$-smooth regular and closed curves $\gamma_i:[0,2\pi]\rightarrow (\partial\Sigma)_i$. Let
$A$ be defined by (3.32) and $d\sigma_\Sigma,~~d\overline{\sigma_\Sigma}$ be defined by (4.5) and $d\overline{s}$ be defined by (4.1)
 and $k^{\infty,s}_{\gamma_i,\Sigma}$ be the sub-Riemannian signed geodesic
curvature of $\gamma_i$ relative to $\Sigma$.  Suppose that the characteristic set $C(\Sigma)$ satisfies $\mathcal{H}^1(C(\Sigma))=0$ where $\mathcal{H}^1(C(\Sigma))$ denotes the Euclidean $1$-dimensional Hausdorff measure of $C(\Sigma)$ and that
$||\nabla_Hu||_H^{-1}$ is locally summable with respect to the Euclidean $2$-dimensional Hausdorff measure
near the characteristic set $C(\Sigma)$, then
\begin{equation}
\int_{\Sigma}\overline{q}^2d\sigma_\Sigma=0,
\end{equation}
\begin{equation}
-\int_{\Sigma}\overline{q}^2d\overline{\sigma_\Sigma}+\int_{\Sigma}Ad\sigma_\Sigma+\sum_{i=1}^n\int_{\gamma_i}k^{\infty,s}_{\gamma_i,\Sigma}d\overline{s}=0.
\end{equation}
\end{thm}
\begin{proof}
Using the discussions in \cite{BTV1}, we know that the number of points satisfying
 $\omega(\dot{\gamma_i}(t))= 0$ and $\frac{d}{dt}(\omega(\dot{\gamma_i}(t)))\neq 0$ on $\gamma_i$ is finite. Since our proof of Theorem 4.3 is based on
 an approximation argument relying on the Lebesgue dominated convergence theorem. In the application of this theorem a set of finite many points can be ignored
 as a null set.
 Then by Lemma 3.6, we have
\begin{equation}
k^{L,s}_{\gamma_i,\Sigma}=k^{\infty,s}_{\gamma_i,\Sigma}+O(L^{-\frac{1}{2}}).
\end{equation}
We assume firstly that $C(\Sigma)$ is empty set.
By the Gauss-Bonnet theorem, we have
\begin{equation}
\int_{\Sigma}\mathcal{K}^{\Sigma,L}\frac{1}{\sqrt{L}}d\sigma_{\Sigma,L}+\sum_{i=1}^n\int_{\gamma_i}k^{L,s}_{\gamma_i,\Sigma}\frac{1}{\sqrt{L}}d{s}_L=2\pi\frac{\chi(\Sigma)}{\sqrt{L}}.
\end{equation}
So by (4.15),(4.16),(4.6),(3.31),(4.3),(4.4), we get
\begin{equation}
-\left(\int_{\Sigma}\overline{q}^2d\sigma_\Sigma\right)L+\left(-\int_{\Sigma}\overline{q}^2d\overline{\sigma_\Sigma}+\int_{\Sigma}Ad\sigma_\Sigma+\sum_{i=1}^n\int_{\gamma_i}k^{\infty,s}_{\gamma_i,\Sigma}d\overline{s}\right)
+O(L^{-\frac{1}{2}})
=2\pi\frac{\chi(\Sigma)}{\sqrt{L}}.
\end{equation}
We multiply (4.17) by a factor $\frac{1}{L}$ and let $L$ go to the infinity and using the dominated convergence theorem, then we get (4.13). Using (4.13) and (4.17), we get (4.14).
Using the similar discussions of the page 27 in \cite{BTV}, we can relax the condition that the characteristic set $C(\Sigma)$ is the empty set and only suppose that the characteristic set $C(\Sigma)$ satisfies $\mathcal{H}^1(C(\Sigma))=0$
and that
$||\nabla_Hu||_H^{-1}$ is locally summable with respect to the Euclidean $2$-dimensional Hausdorff measure
near the characteristic set $C(\Sigma)$.\\
\end{proof}

\section{ The sub-Riemannian limit of curvature of curves in the group of rigid motions of the Minkowski plane}
 We consider the group of rigid motions of the Minkowski plane $E(1,1)$, a unimodular Lie group with a natural subriemannian structure. As a model of $E(1,1)$ we choose the underlying manifold $\mathbb{R}^3$.
  On $\mathbb{R}^3$, we let
\begin{equation}
X_1=\partial_{x_3}, ~~X_2=\frac{1}{\sqrt{2}}(-e^{x_3}\partial_{x_1}+e^{-x_3}\partial_{x_2}),~~X_3=-\frac{1}{\sqrt{2}}(e^{x_3}\partial_{x_1}+e^{-x_3}\partial_{x_2}).
\end{equation}
Then
\begin{equation}
\partial_{x_1}=-\frac{\sqrt{2}}{2}e^{-x_3}(X_2+X_3), ~~\partial_{x_2}=\frac{\sqrt{2}}{2}e^{x_3}(X_2-X_3),~~\partial_{x_3}=X_1,
\end{equation}
and ${\rm span}\{X_1,X_2,X_3\}=T(E(1,1)).$ Let $H={\rm span}\{X_1,X_2\}$ be the horizontal distribution on $E(1,1)$.
Let $\omega_1=dx_3,~~\omega_2=\frac{1}{\sqrt{2}}(-e^{-x_3}dx_1+e^{x_3}dx_2)
,~~\omega=-\frac{1}{\sqrt{2}}(e^{-x_3}dx_1+e^{x_3}dx_2)
.$ Then $H={\rm Ker}\omega$. For the constant $L>0$, let
$g_L=\omega_1\otimes \omega_1+\omega_2\otimes \omega_2+L\omega\otimes \omega,~~g=g_1$ be the Riemannian metric on $E(1,1)$. Then $X_1,X_2,\widetilde{X_3}:=L^{-\frac{1}{2}}X_3$ are orthonormal basis on $T(E(1,1))$ with respect to $g_L$. We have
\begin{equation}
[X_1,X_2]=X_3,~~[X_2,X_3]=0,~~[X_1,X_3]=X_2.
\end{equation}
Let $\nabla^L$ be the Levi-Civita connection on $E(1,1)$ with respect to $g_L$. By the Koszul formula and (5.3), similar to Lemma 2.1, we have
\vskip 0.5 true cm
\begin{lem}
Let $E(1,1)$ be the group of rigid motions of the Minkowski plane, then
\begin{align}
&\nabla^L_{X_j}X_j=0,~~~1\leq j\leq 3,~~~\nabla^L_{X_1}X_2=\frac{L-1}{2L}X_3,~~~ \nabla^L_{X_2}X_1=\frac{-L-1}{2L}X_3,\\
&\nabla^L_{X_1}X_3=\frac{1-L}{2}X_2,~~\nabla^L_{X_3}X_1=\frac{-1-L}{2}X_2,
\nabla^L_{X_2}X_3=\nabla^L_{X_3}X_2=\frac{1+L}{2}X_1.\notag
\end{align}
\end{lem}
\vskip 0.5 true cm
\begin{defn}
Let $\gamma:[a,b]\rightarrow (E(1,1),g_L)$ be a Euclidean $C^1$-smooth curve. We say that
$\gamma(t)$ is a horizontal point of $\gamma$ if
$$\omega(\dot{\gamma}(t))=-\frac{\sqrt{2}}{2}\left(e^{-\gamma_3}\dot{\gamma_1}+e^{\gamma_3}\dot{\gamma_2}\right)=0.$$
\end{defn}
Similar to the definition 2.3 and definition 2.5, we can define $k_\gamma^L$ and $k_\gamma^{\infty}$ for the group of rigid motions of the Minkowski plane, we have
\vskip 0.5 true cm
\begin{lem}
Let $\gamma:[a,b]\rightarrow (E(1,1),g_L)$ be a Euclidean $C^2$-smooth regular curve in the Riemannian manifold $(E(1,1),g_L)$. Then
\begin{equation}
k_{\gamma}^{\infty}=\frac{\sqrt{\frac{1}{2}\left(-e^{-\gamma_3}\dot{\gamma_1}+e^{\gamma_3}\dot{\gamma_2}\right)^2+\dot{\gamma}_3^2}}
{|\omega(\dot{\gamma}(t))|},~~if ~~\omega(\dot{\gamma}(t))\neq 0,
\end{equation}
\begin{align}
&k^{\infty}_{\gamma}=\left\{
\frac{\ddot{\gamma}_3^2+\frac{1}{2}(\ddot{\gamma}_2e^{\gamma_3}+\dot{\gamma}_2\dot{\gamma}_3e^{\gamma_3}-\ddot{\gamma}_1e^{-\gamma_3}+\dot{\gamma}_1\dot{\gamma}_3e^{-\gamma_3})^2}
{\left[\frac{1}{2}\left(-e^{-\gamma_3}\dot{\gamma_1}+e^{\gamma_3}\dot{\gamma_2}\right)^2+\dot{\gamma}_3^2\right]^2}\right.\\\notag
&\left.-\frac{\left[\dot{\gamma}_3\ddot{\gamma}_3+\frac{1}{2}\left(-e^{-\gamma_3}\dot{\gamma_1}+e^{\gamma_3}\dot{\gamma_2}\right)
(\ddot{\gamma}_2e^{\gamma_3}+\dot{\gamma}_2\dot{\gamma}_3e^{\gamma_3}-\ddot{\gamma}_1e^{-\gamma_3}+\dot{\gamma}_1\dot{\gamma}_3e^{-\gamma_3})\right]^2}
{\left[\frac{1}{2}\left(-e^{-\gamma_3}\dot{\gamma_1}+e^{\gamma_3}\dot{\gamma_2}\right)^2+\dot{\gamma}_3^2\right]^3}\right\}^{\frac{1}{2}}\\\notag
&~~if ~~\omega(\dot{\gamma}(t))= 0 ~~and~~\frac{d}{dt}(\omega(\dot{\gamma}(t)))=0,\\\notag
\end{align}
\begin{equation}
{\rm lim}_{L\rightarrow +\infty}\frac{k_{\gamma}^{L}}{\sqrt{L}}=\frac{|\frac{d}{dt}(\omega(\dot{\gamma}(t)))|}{\frac{1}{2}\left(-e^{-\gamma_3}\dot{\gamma_1}+e^{\gamma_3}\dot{\gamma_2}\right)^2+\dot{\gamma}_3^2},~~if ~~\omega(\dot{\gamma}(t))= 0
~~and~~\frac{d}{dt}(\omega(\dot{\gamma}(t)))\neq 0.
\end{equation}
\end{lem}
\begin{proof}
By (5.2), we have
\begin{equation}
\dot{\gamma}(t)=\dot{\gamma}_3X_1+\frac{\sqrt{2}}{2}\left(-e^{-\gamma_3}\dot{\gamma_1}+e^{\gamma_3}\dot{\gamma_2}\right)X_2+\omega(\dot{\gamma}(t))X_3.
\end{equation}
By Lemma 5.1 and (5.8), we have
\begin{align}
&\nabla^L_{\dot{\gamma}}X_1=-\frac{L+1}{2}\omega(\dot{\gamma}(t))X_2-\frac{\sqrt{2}}{2}\left(-e^{-\gamma_3}\dot{\gamma_1}+e^{\gamma_3}\dot{\gamma_2}\right)\frac{L+1}{2L}X_3
,\\\notag
&\nabla^L_{\dot{\gamma}}X_2=\frac{L+1}{2}\omega(\dot{\gamma}(t))X_1+\frac{L-1}{2L}\dot{\gamma}_3X_3
,\\ \notag
&\nabla^L_{\dot{\gamma}}X_3=\frac{\sqrt{2}}{4}(L+1)\left(-e^{-\gamma_3}\dot{\gamma_1}+e^{\gamma_3}\dot{\gamma_2}\right)X_1+\frac{1-L}{2}\dot{\gamma}_3X_2.
\notag
\end{align}
By (5.8) and (5.9), we have
\begin{align}
\nabla^L_{\dot{\gamma}}\dot{\gamma}&=
\left[\ddot{\gamma}_3+\frac{\sqrt{2}}{2}(L+1)\left(-e^{-\gamma_3}\dot{\gamma_1}+e^{\gamma_3}\dot{\gamma_2}\right)\omega(\dot{\gamma}(t))\right]X_1\\\notag
&+\left[\frac{\sqrt{2}}{2}(\ddot{\gamma}_2e^{\gamma_3}+\dot{\gamma}_2\dot{\gamma}_3e^{\gamma_3}-\ddot{\gamma}_1e^{-\gamma_3}+\dot{\gamma}_1\dot{\gamma}_3e^{-\gamma_3})-L\omega(\dot{\gamma}(t))\dot{\gamma}_3\right]X_2
\\\notag
&+\left[\frac{d}{dt}(\omega(\dot{\gamma}(t)))-\frac{\sqrt{2}}{2L}\left(-e^{-\gamma_3}\dot{\gamma_1}+e^{\gamma_3}\dot{\gamma_2}\right)\dot{\gamma}_3\right]X_3.\notag
\end{align}
By (5.8) and (5.10), when $\omega(\dot{\gamma}(t))\neq 0$, we have
$$||\nabla^L_{\dot{\gamma}}{\dot{\gamma}}||_L^2\sim \left[\frac{1}{2}\left(-e^{-\gamma_3}\dot{\gamma_1}+e^{\gamma_3}\dot{\gamma_2}\right)^2+\dot{\gamma}_3^2\right]\omega(\dot{\gamma}(t))^2L^2,~~as~~L\rightarrow +\infty,$$
$$||\dot{\gamma}||^2_L\sim L\omega(\dot{\gamma}(t))^2,~~as~~L\rightarrow +\infty,$$
$$\langle \nabla^L_{\dot{\gamma}}{\dot{\gamma}},\dot{\gamma}\rangle^2_L\sim O(L^2)~~as~~L\rightarrow +\infty.$$
Therefore
$$\frac{||\nabla^L_{\dot{\gamma}}{\dot{\gamma}}||_L^2}{||\dot{\gamma}||^4_L}\rightarrow \frac{\frac{1}{2}\left(-e^{-\gamma_3}\dot{\gamma_1}+e^{\gamma_3}\dot{\gamma_2}\right)^2+\dot{\gamma}_3^2}
{\omega(\dot{\gamma}(t))^2},~~as~~L\rightarrow +\infty,$$
$$\frac{\langle \nabla^L_{\dot{\gamma}}{\dot{\gamma}},\dot{\gamma}\rangle^2_L}{||\dot{\gamma}||^6_L}\rightarrow 0,~~as~~L\rightarrow +\infty.$$
So by (2.6), we have (5.5).
(5.6) comes from (5.8),(5.10),(2.6) and $\omega(\dot{\gamma}(t))= 0$ and $\frac{d}{dt}(\omega(\dot{\gamma}(t)))=0$.
When
$\omega(\dot{\gamma}(t))= 0$
~~and~~$\frac{d}{dt}(\omega(\dot{\gamma}(t)))\neq 0$, we have
$$||\nabla^L_{\dot{\gamma}}{\dot{\gamma}}||_L^2\sim L[\frac{d}{dt}(\omega(\dot{\gamma}(t)))]^2,~~as~~L\rightarrow +\infty,$$
$$||\dot{\gamma}||^2_L=\frac{1}{2}\left(-e^{-\gamma_3}\dot{\gamma_1}+e^{\gamma_3}\dot{\gamma_2}\right)^2+\dot{\gamma}_3^2,$$
$$\langle \nabla^L_{\dot{\gamma}}{\dot{\gamma}},\dot{\gamma}\rangle^2_L=O(1)~~as~~L\rightarrow +\infty.$$
By (2.6), we get (5.7).
\end{proof}
\vskip 0.5 true cm
\section{The sub-Riemannian limit of geodesic curvature of curves on surfaces in the group of rigid motions of the Minkowski plane}
\indent We will consider a regular surface $\Sigma_1\subset(E(1,1),g_L)$ and regular curve $\gamma\subset\Sigma_1$. We will assume that there exists
a Euclidean $C^2$-smooth function $u:E(1,1)\rightarrow \mathbb{R}$ such that
$$\Sigma_1=\{(x_1,x_2,x_3)\in E(1,1):u(x_1,x_2,x_3)=0\}.$$
Similar to Section 3, we define $p,q,r,l,l_L,\overline{p},\overline{q},\overline{p_L},\overline{q_L},\overline{r_L},v_L,e_1,e_2,J_L,k_{\gamma,\Sigma_1}^L,
k_{\gamma,\Sigma_1}^{\infty},k_{\gamma,\Sigma_1}^{L,s},k_{\gamma,\Sigma_1}^{\infty,s}$.
By (3.4) and (5.10), we have
\begin{align}
\nabla^{\Sigma_1,L}_{\dot{\gamma}}\dot{\gamma}&=
\left\{\overline{q}\left[\ddot{\gamma}_3+\frac{\sqrt{2}}{2}(L+1)\left(-e^{-\gamma_3}\dot{\gamma_1}+e^{\gamma_3}\dot{\gamma_2}\right)\omega(\dot{\gamma}(t))\right]
\right.\\\notag
&\left.-\overline{p}\left[\frac{\sqrt{2}}{2}(\ddot{\gamma}_2e^{\gamma_3}+\dot{\gamma}_2\dot{\gamma}_3e^{\gamma_3}-\ddot{\gamma}_1e^{-\gamma_3}+\dot{\gamma}_1
\dot{\gamma}_3e^{-\gamma_3})-L\omega(\dot{\gamma}(t))\dot{\gamma}_3\right]\right\}e_1\\\notag
&+\left\{\overline{r_L}~\overline{p}\left[\ddot{\gamma}_3+\frac{\sqrt{2}}{2}(L+1)\left(-e^{-\gamma_3}\dot{\gamma_1}+e^{\gamma_3}\dot{\gamma_2}\right)\omega(\dot{\gamma}(t))\right]
\right.\\\notag
&+\overline{r_L}~\overline{q}\left[\frac{\sqrt{2}}{2}(\ddot{\gamma}_2e^{\gamma_3}+\dot{\gamma}_2\dot{\gamma}_3e^{\gamma_3}-\ddot{\gamma}_1e^{-\gamma_3}+\dot{\gamma}_1
\dot{\gamma}_3e^{-\gamma_3})-L\omega(\dot{\gamma}(t))\dot{\gamma}_3\right]\\\notag
&\left.-\frac{l}{l_L}L^{\frac{1}{2}}\left[\frac{d}{dt}(\omega(\dot{\gamma}(t)))-\frac{\sqrt{2}}{2L}\left(-e^{-\gamma_3}\dot{\gamma_1}+e^{\gamma_3}\dot{\gamma_2}\right)\dot{\gamma}_3\right]
\right\}e_2.\notag
\end{align}
By (5.8) and $\dot{\gamma}(t)\in T\Sigma_1$, we have
\begin{equation}
\dot{\gamma}(t)=\left[\overline{q}\dot{\gamma}_3-\frac{\sqrt{2}}{2}\overline{p}\left(-e^{-\gamma_3}\dot{\gamma_1}+e^{\gamma_3}\dot{\gamma_2}\right)\right]e_1
-\frac{l_L}{l}L^{\frac{1}{2}}\omega(\dot{\gamma}(t))e_2.
\end{equation}
We have
\vskip 0.5 true cm
\begin{lem}
Let $\Sigma_1\subset(E(1,1),g_L)$ be a regular surface.
Let $\gamma:[a,b]\rightarrow \Sigma_1$ be a Euclidean $C^2$-smooth regular curve. Then
\begin{equation}
k_{\gamma,\Sigma_1}^{\infty}=\frac{\sqrt{\frac{1}{2}\overline{q}^2\left(-e^{-\gamma_3}\dot{\gamma_1}+e^{\gamma_3}\dot{\gamma_2}\right)^2+\overline{p}^2\dot{\gamma_3}^2}
}{|\omega(\dot{\gamma}(t))|},~~if ~~\omega(\dot{\gamma}(t))\neq 0,
\end{equation}
$$k^{\infty}_{\gamma,\Sigma_1}=0,
~~if ~~\omega(\dot{\gamma}(t))= 0, ~~and~~\frac{d}{dt}(\omega(\dot{\gamma}(t)))=0,$$
\begin{equation}
{\rm lim}_{L\rightarrow +\infty}\frac{k_{\gamma,\Sigma_1}^{L}}{\sqrt{L}}=\frac{|\frac{d}{dt}(\omega(\dot{\gamma}(t)))|}
{\left[\overline{q}{\dot{\gamma}_3}-\frac{\sqrt{2}}{2}\overline{p}\left(-e^{-\gamma_3}\dot{\gamma_1}+e^{\gamma_3}\dot{\gamma_2}\right)\right]^2}
,~~if ~~\omega(\dot{\gamma}(t))= 0
~~and~~\frac{d}{dt}(\omega(\dot{\gamma}(t)))\neq 0.
\end{equation}
\end{lem}
\begin{proof}
By (6.1), we have
\begin{align}
||\nabla^{\Sigma_1,L}_{\dot{\gamma}}\dot{\gamma}||^2_{L,\Sigma_1}
\sim L^2\omega(\dot{\gamma}(t))^2\left[\frac{1}{2}\overline{q}^2\left(-e^{-\gamma_3}\dot{\gamma_1}+e^{\gamma_3}\dot{\gamma_2}\right)^2+\overline{p}^2\dot{\gamma_3}^2\right],~~~~ {\rm as} ~~~~L\rightarrow +\infty.
\end{align}
By (6.2), we have that when $\omega(\dot{\gamma}(t))\neq 0$,
\begin{equation}
||\dot{\gamma}||_{\Sigma_1,L}\sim L^{\frac{1}{2}}|\omega(\dot{\gamma}(t))|
,~~~~ {\rm as} ~~~~L\rightarrow +\infty.
\end{equation}
By (6.1) and (6.2), we have
\begin{align}
\langle \nabla^{\Sigma_1,L}_{\dot{\gamma}}{\dot{\gamma}},\dot{\gamma}\rangle_{\Sigma_1,L}
\sim M_0L,
\end{align}
where $M_0$ does not depend on $L$.
By (3.7),(6.5)-(6.7), we get (6.3).
When $\omega(\dot{\gamma}(t))= 0$ and $\frac{d}{dt}(\omega(\dot{\gamma}(t)))=0$, we have
\begin{align}
||\nabla^{\Sigma_1,L}_{\dot{\gamma}}\dot{\gamma}||^2_{L,\Sigma_1}\sim
\left[\overline{q}\ddot{\gamma}_3
-\frac{\sqrt{2}}{2}\overline{p}(\ddot{\gamma}_2e^{\gamma_3}+\dot{\gamma}_2\dot{\gamma}_3e^{\gamma_3}-\ddot{\gamma}_1e^{-\gamma_3}+\dot{\gamma}_1
\dot{\gamma}_3e^{-\gamma_3})\right]^2
,~~~~ {\rm as} ~~~~L\rightarrow +\infty,
\end{align}
and
\begin{equation}
||\dot{\gamma}||_{\Sigma_1,L}^2=\left[\overline{q}{\dot{\gamma}_3}-\frac{\sqrt{2}}{2}\overline{p}\left(-e^{-\gamma_3}\dot{\gamma_1}+e^{\gamma_3}\dot{\gamma_2}\right)\right]^2,
\end{equation}
\begin{align}
\langle \nabla^{\Sigma_1,L}_{\dot{\gamma}}{\dot{\gamma}},\dot{\gamma}\rangle_{\Sigma_1,L}=&
\left[\overline{q}\ddot{\gamma}_3
-\frac{\sqrt{2}}{2}\overline{p}(\ddot{\gamma}_2e^{\gamma_3}+\dot{\gamma}_2\dot{\gamma}_3e^{\gamma_3}-\ddot{\gamma}_1e^{-\gamma_3}+\dot{\gamma}_1
\dot{\gamma}_3e^{-\gamma_3})\right]\\\notag
&\cdot\left[\overline{q}{\dot{\gamma}_3}-\frac{\sqrt{2}}{2}\overline{p}\left(-e^{-\gamma_3}\dot{\gamma_1}+e^{\gamma_3}\dot{\gamma_2}\right)\right].\notag
\end{align}
By (6.8)-(6.10) and (3.7), we get $k^{\infty}_{\gamma,\Sigma_1}=0$.
When $\omega(\dot{\gamma}(t))= 0$ and $\frac{d}{dt}(\omega(\dot{\gamma}(t)))\neq 0$, we have
$$||\nabla^{\Sigma_1,L}_{\dot{\gamma}}\dot{\gamma}||^2_{L,\Sigma_1}\sim L
[\frac{d}{dt}(\omega(\dot{\gamma}(t)))]^2,$$
$$\langle \nabla^{\Sigma_1,L}_{\dot{\gamma}}{\dot{\gamma}},\dot{\gamma}\rangle_{\Sigma_1,L}=O(1),$$
so we get (6.4).
\end{proof}
\vskip 0.5 true cm
\begin{lem}
Let $\Sigma_1\subset(E(1,1),g_L)$ be a regular surface.
Let $\gamma:[a,b]\rightarrow \Sigma_1$ be a Euclidean $C^2$-smooth regular curve. Then
\begin{equation}
k_{\gamma,\Sigma_1}^{\infty,s}=\frac{\overline{p}{\dot{\gamma}_3}+\frac{\sqrt{2}}{2}\overline{q}\left(-e^{-\gamma_3}\dot{\gamma_1}+e^{\gamma_3}\dot{\gamma_2}\right)}
{|\omega(\dot{\gamma}(t))|},~~if ~~\omega(\dot{\gamma}(t))\neq 0,
\end{equation}
$$k^{\infty,s}_{\gamma,\Sigma_1}=0
~~if ~~\omega(\dot{\gamma}(t))= 0, ~~and~~\frac{d}{dt}(\omega(\dot{\gamma}(t)))=0,$$
\begin{align}
{\rm lim}_{L\rightarrow +\infty}\frac{k_{\gamma,\Sigma_1}^{L,s}}{\sqrt{L}}=&\frac{\left[-\overline{q}{\dot{\gamma}_3}+\frac{\sqrt{2}}{2}\overline{p}\left(-e^{-\gamma_3}\dot{\gamma_1}+e^{\gamma_3}\dot{\gamma_2}\right)\right]
\frac{d}{dt}(\omega(\dot{\gamma}(t)))
}{\left|\overline{q}{\dot{\gamma}_3}-\frac{\sqrt{2}}{2}\overline{p}\left(-e^{-\gamma_3}\dot{\gamma_1}+e^{\gamma_3}\dot{\gamma_2}\right)\right|^3},\\\notag
&~~if ~~\omega(\dot{\gamma}(t))= 0
~~and~~\frac{d}{dt}(\omega(\dot{\gamma}(t)))\neq 0.\notag
\end{align}
\end{lem}
\begin{proof} By (3.3) and (6.2), we have
\begin{equation}
J_L(\dot{\gamma})=\frac{l_L}{l}L^{\frac{1}{2}}\omega(\dot{\gamma}(t))e_1+\left[\overline{q}\dot{\gamma}_3-\frac{\sqrt{2}}{2}\overline{p}\left(-e^{-\gamma_3}\dot{\gamma_1}+e^{\gamma_3}\dot{\gamma_2}\right)\right]e_2
.
\end{equation}
By (6.1) and (6.13), we have
\begin{align}
\langle \nabla^{\Sigma_1,L}_{\dot{\gamma}}\dot{\gamma},J_L(\dot{\gamma})\rangle_{L,\Sigma_1}
\sim L^{\frac{3}{2}}\omega(\dot{\gamma}(t))^2\left[\overline{p}\dot{\gamma}_3+\frac{\sqrt{2}}{2}\overline{q}\left(-e^{-\gamma_3}\dot{\gamma_1}+e^{\gamma_3}\dot{\gamma_2}\right)\right],
~~{\rm as}~~L\rightarrow +\infty.
\end{align}
So by (3.17),(6.6) and (6.14),
we get (6.11).
When $\omega(\dot{\gamma}(t))= 0$ and $\frac{d}{dt}(\omega(\dot{\gamma}(t)))=0,$
we get
\begin{align}
\langle \nabla^{\Sigma_1,L}_{\dot{\gamma}}\dot{\gamma},J_L(\dot{\gamma})\rangle_{L,\Sigma_1}
\sim M_0L^{-\frac{1}{2}}~~{\rm as}~~L\rightarrow +\infty.
\end{align}
So $k^{\infty,s}_{\gamma,\Sigma_1}=0.$
When $\omega(\dot{\gamma}(t))= 0$ and $\frac{d}{dt}(\omega(\dot{\gamma}(t)))\neq 0$, we have
\begin{align}
\langle \nabla^{\Sigma_1,L}_{\dot{\gamma}}\dot{\gamma},J_L(\dot{\gamma})\rangle_{L,\Sigma_1}\sim
L^{\frac{1}{2}}\left[-\overline{q}{\dot{\gamma}_3}+\frac{\sqrt{2}}{2}\overline{p}\left(-e^{-\gamma_3}\dot{\gamma_1}+e^{\gamma_3}\dot{\gamma_2}\right)\right]
\frac{d}{dt}(\omega(\dot{\gamma}(t))),~~{\rm as}~~L\rightarrow +\infty.
\end{align}
So we get (6.12).
\end{proof}

In the following, we compute the sub-Riemannian limit of the Riemannian Gaussian curvature of surfaces in the group of rigid motions of the Minkowski plane.
Similarly to Theorem 4.3 in \cite{CDPT}, we have
\vskip 0.5 true cm
\begin{thm} The second fundamental form $II^L_1$ of the
embedding of $\Sigma_1$ into $(E(1,1),g_L)$ is given by
\begin{equation}
II^L_1=\left(
  \begin{array}{cc}
   h_{11}, & h_{12}\\
   h_{21}, & h_{22}\\
  \end{array}
\right),
\end{equation}
where $$h_{11}= \frac{l}{l_L}[X_1(\overline{p})+X_2(\overline{q})]-\overline{p}~\overline{q}~\overline{r_L}L^{-\frac{1}{2}},$$
    $$h_{12}=h_{21}=-\frac{l_L}{l}\langle e_1,\nabla_H(\overline{r_L})\rangle_L-\frac{\sqrt{L}}{2}+\frac{1}{2\sqrt{L}}(\overline{q_L}^2-\overline{p_L}^2)+\frac{1}{2\sqrt{L}}\overline{r_L}^2(\overline{q}^2-\overline{p}^2),$$
    $$h_{22}=-\frac{l^2}{l_L^2}\langle e_2,\nabla_H(\frac{r}{l})\rangle_L+\widetilde{X_3}(\overline{r_L})+\overline{p_L}~\overline{q_L}~\overline{r_L} L^{-\frac{1}{2}}+
    \overline{p}~\overline{q}~\overline{r_L}^3 L^{-\frac{1}{2}}.$$
\end{thm}
\vskip 0.5 true cm
Similar to Proposition 3.8, we have
\vskip 0.5 true cm
\begin{prop} Away from characteristic points, the horizontal mean curvature $\mathcal{H}_{\infty}^1$ of $\Sigma_1\subset  E(1,1)$ is given by
\begin{equation}
\mathcal{H}_{\infty}^1=X_1(\overline{p})+X_2(\overline{q}).
\end{equation}
\end{prop}
\vskip 0.5 true cm
By Lemma 5.1, we have
\vskip 0.5 true cm
\begin{lem}
Let $E(1,1)$ be the group of rigid motions of the Minkowski plane, then
\begin{align}
&
R^L(X_1,X_2)X_1=(\frac{1}{2}-\frac{1}{4L}+\frac{3L}{4})X_2,~~~ R^L(X_1,X_2)X_2=(-\frac{1}{2}+\frac{1}{4L}-\frac{3L}{4})X_1,\\\notag
&~~~ R^L(X_1,X_2)X_3=0,~~~
R^L(X_1,X_3)X_1=(\frac{1}{2}-\frac{L}{4}+\frac{3}{4L})X_3,\\\notag
&~~~ R^L(X_1,X_3)X_2=0,
~~~ R^L(X_1,X_3)X_3=(\frac{L^2}{4}-\frac{L}{2}-\frac{3}{4})X_1,~~~\\\notag
& R^L(X_2,X_3)X_1=0,~~~
R^L(X_2,X_3)X_2=-(\frac{1}{2}+\frac{1}{4L}+\frac{L}{4})X_3,\\\notag
&~~~ R^L(X_2,X_3)X_3=(\frac{1}{4}+\frac{L^2}{4}+\frac{L}{2})X_2.\notag
\end{align}
\end{lem}
\vskip 0.5 true cm
\begin{prop} Away from characteristic points, we have
\begin{equation}
\mathcal{K}^{\Sigma_1,\infty}(e_1,e_2)=-\langle e_1,\nabla_H(\frac{X_3u}{|\nabla_Hu|})\rangle-\frac{(X_3u)^2}{l^2}.
\end{equation}
\end{prop}
\begin{proof} By (3.33) and Lemma 6.5, we have
\begin{align}
\mathcal{K}^{E(1,1),L}(e_1,e_2)=&-\overline{r_L}^2(\frac{1}{2}-\frac{1}{4L}+\frac{3L}{4})-\left(\frac{l}{l_L}\overline{q}\right)^2(\frac{1}{2}-\frac{L}{4}+\frac{3}{4L})\\\notag
&+\left(\frac{l}{l_L}\overline{p}\right)^2(\frac{1}{2}+\frac{1}{4L}+\frac{L}{4})\\\notag
&\sim \left(\frac{l}{l_L}\right)^2\frac{L}{4}-\frac{3}{4}\frac{(X_3u)^2}{l^2}-\overline{q}^2+\frac{1}{2},~~~~ {\rm as} ~~~~L\rightarrow +\infty.\notag
\end{align}
Similar to (3.35), we have
\begin{equation}
{\rm det}(II^L_1)=-\frac{L}{4}-\langle e_1,\nabla_H(\frac{X_3u}{|\nabla_Hu|})\rangle+\frac{1}{2}(\overline{q}^2-\overline{p}^2)
+O(L^{-\frac{1}{2}})~~{\rm as}~~L\rightarrow +\infty.
\end{equation}
By (6.21) and (6.22), we have (6.20).
\end{proof}
Similar to (4.3) and (4.6), for the group of rigid motions of the Minkowski plane, we have
\begin{equation}
{\rm lim}_{L\rightarrow +\infty}\frac{1}{\sqrt{L}}ds_L= ds,~~~~
{\rm lim}_{L\rightarrow +\infty}\frac{1}{\sqrt{L}}d\sigma_{\Sigma_1,L}=d\sigma_{\Sigma_1}.
\end{equation}
By (6.20),(6.23) and Lemma 6.2, similar to the proof of Theorem 1 in \cite{BTV}, we have
\vskip 0.5 true cm
\begin{thm}
 Let $\Sigma_1\subset (E(1,1),g_L)$
  be a regular surface with finitely many boundary components $(\partial\Sigma_1)_i,$ $i\in\{1,\cdots,n\}$, given by Euclidean $C^2$-smooth regular and closed curves $\gamma_i:[0,2\pi]\rightarrow (\partial\Sigma_1)_i$.
 Suppose that the characteristic set $C(\Sigma_1)$ satisfies $\mathcal{H}^1(C(\Sigma_1))=0$ and that
$||\nabla_Hu||_H^{-1}$ is locally summable with respect to the Euclidean $2$-dimensional Hausdorff measure
near the characteristic set $C(\Sigma_1)$, then
\begin{equation}
\int_{\Sigma_1}\mathcal{K}^{\Sigma_1,\infty}d\sigma_{\Sigma_1}+\sum_{i=1}^n\int_{\gamma_i}k^{\infty,s}_{\gamma_i,\Sigma_1}d{s}=0.
\end{equation}
\end{thm}

\vskip 1 true cm

\section{Acknowledgements}

The first author was supported in part by  NSFC No.11771070. The authors are deeply grateful to the referees for their valuable comments and helpful suggestions.

\vskip 1 true cm


\bigskip

\noindent {\footnotesize {\it Yong Wang} \\
{School of Mathematics and Statistics, Northeast Normal University, Changchun 130024, China}\\
{Email: wangy581@nenu.edu.cn}
\vskip 0.5 true cm
\noindent {\footnotesize {\it Sining Wei} \\
{School of Mathematics and Statistics, Northeast Normal University, Changchun 130024, China}\\
{Email: weisn835@nenu.edu.cn}
\end{document}